\documentclass[10pt]{amsart}
\usepackage{amssymb,amsbsy,amsmath,amsfonts,amssymb}
\usepackage{latexsym,euscript,exscale}

\title{An exact Ramsey principle for block sequences}
\author {Christian Rosendal}
\date {June 2008}
\newcommand {\A}{\mathbb A}
\newcommand {\B}{\mathbb B}

\newcommand {\N}{\mathbb N}

\newcommand {\R}{\mathbb R}

\newcommand {\C}{\mathbb C}
\newcommand {\U}{\mathbb U}

\newcommand {\D}{\mathbb D}

\newcommand{\norm}[1]{\lVert#1\rVert}

\newcommand{\om}{\omega}
\newcommand{\eps}{\epsilon}

\newcommand{\con}{\;\hat{}\;}

\newcommand{\tom} {\emptyset}

\newcommand{\equi}{\Leftrightarrow}

\newcommand{\til}{\rightarrow}

\newcommand{\Lim}[1]{\mathop{\longrightarrow}\limits_{#1}}

\newcommand {\del}{ \; \big| \;}

\newcommand {\go} {\mathfrak}
\newcommand {\ku} {\mathcal}

\newcommand {\e} {\exists}
\renewcommand {\a} {\forall}

\newtheorem{thm}{Theorem}
\newtheorem{cor}[thm]{Corollary}
\newtheorem{lemme}[thm]{Lemma}
\newtheorem{souslemme}[thm]{Sublemma}

\newtheorem{defi} [thm] {Definition}

\begin{document}

\thanks{The author was partially supported by NSF grant DMS 0556368}



\maketitle

\begin{abstract}
We prove an exact, i.e., formulated without $\Delta$-expansions, Ramsey
principle for infinite block sequences in vector spaces over countable fields,
where the two sides of the dichotomic principle are represented by respectively
winning strategies in Gowers' block sequence game and winning strategies in the
infinite asymptotic game. This allows us to recover Gowers' dichotomy theorem
for block sequences in normed vector spaces by a simple application of the
basic determinacy theorem for infinite asymptotic games.
\end{abstract}

\section{Introduction}
The results presented here represent a new approach to the fundamental result
of W.T. Gowers \cite{gowers}, whose uses in Banach space theory seem far from
exhausted (for applications see, e.g., \cite{gowers,minimal}). Gowers' result
is a Ramsey theoretic statement for Banach spaces that cleverly combines Ramsey
theory and game theory to compensate for the fact that a true Ramsey theoretic
result fails to hold in general. The proof of Gowers' theorem, however,
involves approximation arguments that significantly cloud the main ideas and
lead to very unwieldy computations, as can be seen from the existing proofs
\cite{gowers,jordi,kalton}. Moreover, the notion of {\em weakly Ramsey sets}
extracted from the proof incorporates approximations, which makes it hard to
induct over and extend beyond the class of analytic sets. For example, it was
unknown whether ${\bf \Sigma}^1_2$ sets are weakly Ramsey assuming Martin's
axiom, though it was shown to hold under a strengthening of MA by J. Bagaria
and J. L\'opez-Abad \cite{jordi}.

The novelty of our approach lies in the replacement of both sides of the
dichotomy with game theoretical statements, which completely eschew
approximations and allow for a very simple inductive proof. The new tools are
the {\em infinite asymptotic game} and the definition of {\em strategically
Ramsey sets} in vector spaces over countable fields. Using these, one easily
shows that under MA, ${\bf \Sigma}^1_2$ sets are strategically Ramsey, and a
version of the basic determinacy result for infinite asymptotic games
\cite{iag} connects the notions of weakly Ramsey and strategically Ramsey sets.

\section{Notation}
Let $\go F$ be a countable field and let $E$ be a countably dimensional $\go
F$-vector space with basis $(e_n)$. We equip $E$ with the discrete topology and
its countable power $E^\infty$ with the product topology. Since $E$ is a
countable, $E^\infty$ is a Polish space. Let $x,y,z,v$ be variables for {\em
non-zero} elements of $E$. If $x=\sum a_ne_n\in E$, let ${\rm supp}\; x=\{n\del
a_n\neq 0\}$ and set for $x,y\in E$,
$$
x<y\equi \a n\in {\rm supp}\; x\; \a m\in {\rm supp}\; y\;\; n<m.
$$
Similarly, if $k$ is a natural number, we set
$$
k<x\equi \a n\in {\rm supp}\; x\;\;k<n.
$$
Analogous notation is used for finite subsets of $\N$. A finite or infinite
sequence $(x_0,x_1,x_2,x_3,\ldots)$ of vectors is said to be a block sequence
if for all $n$, $x_n<x_{n+1}$.

Notice that, by elementary linear algebra, for all infinite dimensional
subspaces $X\subseteq E$ there is a subspace $Y\subseteq X$ spanned by an
infinite block sequence, called a {\em block subspace}. Henceforth, we use
variables $X,Y,Z,V,W$ to denote infinite dimensional block subspaces of $E$.
Also, denote infinite block sequences by variables $\bf x,y,z$ and finite block
sequences by variables $\vec x, \vec y, \vec z$.

\section{Gowers' game and the infinite asymptotic game}
Suppose $X\subseteq E$. We define  {\em Gowers' game} $G_X$ played below $X$
between two players I and II as follows: I and II alternate (with I beginning)
in choosing respectively infinite dimensional subspaces  $Y_0, Y_1,
Y_2,\ldots\subseteq X$ and vectors $x_0<x_1<x_2<\ldots$ according to the
constraint $x_i\in Y_i$:
$$
\begin{array}{cccccccccccc}
{\bf I} & & & Y_0 &  & Y_1 &  & Y_2 & & Y_3& &\ldots \\
{\bf II} & & &  & x_0 &  & x_1 &  & x_2 & & x_3 &\ldots
\end{array}
$$
Also,  the {\em infinite asymptotic game} $F_X$ played below $X$ is defined as
follows: I and II alternate (with I beginning) in choosing respectively natural
numbers $n_0< n_1< n_2<\ldots$ and vectors $x_0<x_1<x_2<\ldots\in X$ according
to the constraint $n_i< x_i$:
$$
\begin{array}{cccccccccccc}
{\bf I} & & & n_0 &  & n_1 &  & n_2 & & n_3& &\ldots \\
{\bf II} & & &  & x_0 &  & x_1 &  & x_2 & & x_3 &\ldots
\end{array}
$$
In both games we say that the sequence $(x_n)_{n\in \N}$ is the {\em outcome}
of the game. Moreover, if $\vec x$ is a finite block sequence, we define
Gowers' game $G_X(\vec x)$ and the infinite asymptotic game $F_X(\vec x)$ as
above except that the outcome is now $\vec x\con (x_0,x_1,x_2,\ldots)$.

\

If $X$ and $Y$ are subspaces, where $Y$ is spanned by an infinite block
sequence ${\bf y}=(y_0,y_1,y_2,\ldots)$, we write $Y\subseteq^* X$ if there is
$n$ such that $y_m\in X$ for all $m\geq n$. A simple diagonalisation argument
shows that if $X_0\supseteq X_1\supseteq X_2\supseteq \ldots$ is a decreasing
sequence of block subspaces, then there is some $Y\subseteq X_0$ such that
$Y\subseteq^* X_n$ for all $n$.

The aim of the games above is for each of the players to ensure that the
outcome ${\bf x}$ lies in some predetermined set depending on the player. By
the asymptotic nature of the game, it is easily seen that if $\A\subseteq
E^\infty$ and $Y\subseteq^* X$, then if II has a strategy in $G_X$ to play in
$\A$, i.e., to ensure that the outcome is in $\A$, then II will have a strategy
in $G_Y$ to play in $\A$ too. Similarly, if I has a strategy in $F_X$ to play
in $\A$, then I also has a strategy in $F_Y$ to play in $\A$.

\begin{defi}
We say that a set $\A\subseteq E^\infty$ is {\em strategically Ramsey} if for
all $V\subseteq E$ and all $\vec z$, there is $W\subseteq V$ such that either
\begin{itemize}
  \item [(a)] II has a strategy in $G_W(\vec z)$ to play in $\A$,
  or
  \item [(b)] I has a strategy in $F_W(\vec z)$ to
  play in $\sim\A$.
\end{itemize}
\end{defi}

\section{Analytic sets are strategically Ramsey}

\begin{lemme}\label{open}
Open sets $\U\subseteq E^\infty$ are strategically Ramsey.
\end{lemme}

\begin{proof}
Noticing that for all open $\U$,  $\U^V_{\vec z}=\{(x_i)\in V^\infty\del \vec
z\con (x_i)\in \U\}$ is also an open subset of $V^\infty$, we can suppose $V=E$
and $\vec z=\tom$. We say that
\begin{enumerate}
  \item $(\vec x,X)$ is {\em good} if II has a
strategy in $G_X(\vec x)$ to play in $\U$,
  \item $(\vec x,X)$ is {\em bad} if $\a Y\subseteq X$, $(\vec x,Y)$ is not
  good,
  \item $(\vec x,X)$ is {\em worse} if it is bad and $\e n\; \a y\in X\;( n<y\til (\vec x\con y,X)$ is bad).
\end{enumerate}
We notice that the properties good, bad and worse are $\subseteq^*$-hereditary,
i.e., if $(\vec x,X)$ is good/bad/worse and $Y\subseteq^*X$, then $(\vec x,Y)$
is good/bad/worse.
\begin{souslemme}
If $(\vec x,X)$ is bad, then there is some $Z\subseteq X$ such that $(\vec
x,Z)$ is worse.
\end{souslemme}

\begin{proof}
Notice that, as good and bad are $\subseteq^*$-hereditary, by diagonalising
over all $\vec y$, we can find some $Y\subseteq X$ such that for all $\vec y$,
$(\vec y,Y)$ is either good or bad. Suppose towards a contradiction that there
is no $Z\subseteq Y$ such that $(\vec x,Z)$ is worse. Then, as $(\vec x,Z)$ is
bad for all $Z\subseteq Y$,
$$
\a Z\subseteq Y\; \e y\in Z\; (\vec x\con y,Z)\textrm{ is not bad}
$$
and hence
$$
\a Z\subseteq Y\; \e y\in Z\; (\vec x\con y,Y)\textrm{ is good}.
$$
In other words, for all $Z\subseteq Y$ there is some $y\in Z$ such that II has
a strategy in $G_Y(\vec x\con y)$ to play in $\U$ and therefore II also has a
strategy in $G_Y(\vec x)$ to play in $\U$, contradicting that $(\vec x,X)$ was
bad.
\end{proof}
Again, using the preceding sublemma and diagonalising, we can find some
$X\subseteq E$ such that for all $\vec y$, either $(\vec y,X)$ is good or
worse. Now, if $(\tom,X)$ is good, II has a strategy in $G_X$ to play in $\U$,
so suppose instead that $(\tom,X)$ is worse. We claim that I has a strategy in
$F_X$ to produce block sequences $(x_0,x_1,x_2,\ldots)$ so that for all $m$,
$(x_0,x_1,\ldots,x_m,X)$ is worse. To see this, suppose that at some point of
the game, $\vec x$ has been played so that $(\vec x,X)$ is worse. Then there is
some $n$ such that for all $y\in X$, if $n<y$, then $(\vec x\con y,X)$ is bad
and hence even worse. Thus, we can let I play $n$. But if I follows this
strategy, then, in particular, for no $m$ can II have a strategy in
$G_X(x_0,\ldots,x_m)$ to play in $\U$ and thus as $\U$ is open,
$(x_0,x_1,x_2,\ldots)\in \sim\U$.  Therefore, I has a strategy in $F_X$ to play
in $\sim \U$.
\end{proof}

\begin{lemme}\label{ctbl split}
Suppose $\A_n\subseteq E^\infty$ and $\B=\bigcup_n\A_n$. Let $\vec x$ and
$X\subseteq E$ be given. Then there is $Z\subseteq X$ such that either
\begin{itemize}
  \item [(a)] II has a strategy in $G_Z$ to play $(z_i)$ such that
  $$
 \e n\;\a V\subseteq Z\; \textrm{I has no strategy in $F_V(\vec x\con (z_0,\ldots,z_{n}))$ to play in $\sim\A_n$},
  $$
  or
  \item[(b)] I has a strategy in $F_Z(\vec x)$ to play in $\sim \B$.
\end{itemize}
\end{lemme}

\begin{proof}
We say that $(\vec y,n)$ {\em accepts} $Y$ if I has a strategy in $F_Y(\vec y)$
to play in $\sim\A_n$. Also, $(\vec y,n)$ {\em rejects} $Y$ if $\a Z\subseteq
Y$, $(\vec y,n)$ does not accept $Z$. Notice that acceptance and rejection are
$\subseteq^*$-hereditary, so there is $Y\subseteq X$ such that for all $\vec y$
and $n$, either $(\vec y,n)$ accepts or rejects $Y$. Set
$$
\D=\{(z_i)\del \e n\;(\vec x\con (z_0,\ldots,z_{n}),n) \textrm{ rejects }Y\}
$$
and notice that $\D$ is open. It follows, by Lemma \ref{open}, that there is
$Z\subseteq Y$ such that either II has a strategy in $G_Z$ to play in $\D$ or I
has a strategy in $F_Z$ to play in $\sim\D$.

In the first case, II has a strategy in $G_Z$ to play $(z_i)$ such that
$$
\e n\; \a V\subseteq Y\textrm{ I has no strategy in $F_V(\vec x\con (z_0,\ldots,z_{n}))$ to play in $\sim\A_n$},
$$
which immediately implies (a). So suppose instead that I has a strategy in
$F_Z$ to play in $\sim\D$, i.e., that I has a strategy in $F_Z$ to play $(z_i)$
such that
$$
\a n\; (\vec x\con(z_0,\ldots,z_{n}),n) \textrm{ accepts }Z.
$$
Thus, I has a strategy $\sigma$ in $F_Z$ to play $(z_i)$ such that for all $n$,
I has a strategy $\sigma_{(z_0,\ldots,z_{n})}$ in $F_Z(\vec
x\con(z_0,\ldots,z_{n}))$ to play in $\sim\A_n$. By successively putting more
and more strategies into play, I thus has a strategy in $F_Z(\vec x)$ to play
in $\bigcap_n\sim\A_n=\sim\B$, which gives us (b). Concretely, if at step
$n+1$, $(z_0,\ldots,z_{n})$ has been played, then I will respond with
$$
\max\{\sigma(z_0,\ldots,z_{n}),\sigma_{(z_0)}(z_1,z_2,\ldots,z_{n}),
\ldots,\sigma_{(z_0,\ldots,z_{n})}(\tom)\}.
$$
It follows that if $(z_i)$ is the outcome of the game, then for all $n$, as II
has responded to a stronger strategy than $\sigma_{(z_0,\ldots,z_{n})}$ when
playing $(z_{n+1},z_{n+2},\ldots)$, we see that $\vec x\con
(z_0,\ldots,z_{n})\con (z_{n+1},z_{n+2},\ldots)\in \sim\A_n$. Therefore, $\vec
x\con(z_i)\in \bigcap_n\sim\A_n$.
\end{proof}

Notice that both conclusions (a) and (b) in Lemma \ref{ctbl split} are
$\subseteq^*$-hereditary in $Z$.

\begin{thm}\label{analytic}
Analytic sets  are strategically Ramsey.
\end{thm}

\begin{proof}Suppose $\A\subseteq E^\infty$ is analytic.
Noticing that for all $V\subseteq E$ and $\vec z$, $\A^V_{\vec z}=\{(x_i)\in
V^\infty\del \vec z\con (x_i)\in \A\}$ is also an analytic subset of
$V^\infty$, we can suppose $V=E$ and $\vec z=\tom$. Let $F\colon\N^\N\til \A$
be a continuous surjection and set for every $s\in \N^{<\N}$, $\A_s=F[N_s]$,
where $N_s=\{\alpha\in \N^\N\del s\subseteq \alpha\}$. We note that
$\A_s=\bigcup_{n\in \N}\A_{s\con n}$. Let $\D(s,\vec x,X)$ be the set
\begin{displaymath}\begin{split}
\{(z_i)\del \e n\; \a W\subseteq X\textrm{ I has no strategy in } F_W(\vec
x\con (z_0,\ldots,z_{n})) \textrm{ to play in }\sim\A_{s\con n}\}.
\end{split}\end{displaymath}
By Lemma \ref{ctbl split}, there is $X\subseteq E$ such that for all $\vec x$
and all $s\in \N^{<\N}$ either
\begin{itemize}
  \item [(a)] II has a strategy in $G_X$ to play in $\D(s,\vec x,X)$,
  or
  \item[(b)] I has a strategy in $F_X(\vec x)$ to play in $\sim \A_s$.
\end{itemize}
Suppose that I has no strategy in $F_X$ to play in $\sim\A=\sim\A_\tom$. We
describe a strategy for II in $G_X$ to play in $\A$.

First, as II has a strategy in $G_X$ to play in $\D(\tom,\tom,X)$, he follows
this strategy until $(z_0,\ldots,z_{n_0})$ has been played such that I does not
have a strategy in $F_X(z_0,\ldots,z_{n_0})$ to play in $\sim\A_{n_0}$.

Thus, by the assumption on $X$, II must have a strategy in $G_X$ to play in
$\D((n_0),(z_0,\ldots,z_{n_0}),X)$. II follows this until further
$(z_{n_0+1},\ldots,z_{n_0+n_1+1})$ has been played such that I does not have a
strategy in $F_X(z_0,\ldots,z_{n_0},z_{n_0+1},\ldots,z_{n_0+n_1+1})$ to play in
$\sim\A_{(n_0,n_1)}$.

By the same reasoning as before, II must have a strategy in $G_X$ to play in
the set $\D((n_0,n_1),(z_0,\ldots,z_{n_0+n_1+1}),X)$. He follows this strategy
until yet another $(z_{n_0+n_1+2},\ldots,z_{n_0+n_1+n_2+2})$ has been played
such that I does not have a strategy in $F_X(z_0,\ldots,z_{n_0+n_1+n_2+2})$ to
play in $\sim\A_{(n_0,n_1,n_2)}$.

Continuing in this way and letting $m_k=(\sum_{j\leq k}n_j)+k$, the outcome of
the game will be a sequence
$$
{\bf z}=(z_0,z_1,z_2,\ldots,z_{m_0},\ldots,z_{m_1},\ldots,z_{m_2},\ldots)
$$
such that for the sequence $\alpha=(n_0,n_1,n_2,\ldots)$ and all $k$, I does
not have a strategy in $F_X(z_0,\ldots,z_{m_k})$ to play in
$\sim\A_{(n_0,n_1,\ldots,n_k)}$. It follows that for all $k$, there must be an
infinite block sequence ${\bf z}_k$ end-extending $(z_0,\ldots,z_{m_k})$ such
that ${\bf z}_k\in \A_{(n_0,n_1,\ldots,n_k)}$. So for some $\beta_k\in
N_{(n_0,n_1,\ldots,n_k)}$, we have $F(\beta_k)={\bf z}_k$. But, by continuity
of $F$, we have $F(\beta_k)\Lim{k\til\infty}F(\alpha)$, while ${\bf
z}_k\Lim{k\til\infty}{\bf z}$, so $F(\alpha)={\bf z}$ and ${\bf z}\in \A$.
Therefore, this describes a strategy for II in $G_X$ to play in $\A$.
\end{proof}

\section{Infinite asymptotic games in normed vector spaces}
Suppose now that $\go F$ is a subfield of $\R$ or $\C$ and $\|\cdot\|$ is a
norm on $E$ taking values in $\go F$. For $X\subseteq E$, denote by $\ku B_X$
the unit ball of $X$ and by $\go B(X)$ the set of block sequences $(x_i)$ of
$X$ with $\norm{x_i}\leq 1$. Also, if $\Delta=(\delta_i)$ is a sequence of
strictly positive real numbers, denoted by $\Delta>0$, and $\A\subseteq
E^\infty$, we let
$$
\A_\Delta=\{(z_i)\in E^\infty\del \e (x_i)\in \A\; \a
i\;\|x_i-z_i\|<\delta_i\}.
$$

To get a stronger statement in (b) of the definition of strategically Ramsey
sets, we need to allow approximations. For this, we use a variant of a result
from \cite{iag}, though the proof given here is in the same spirit as that
presented in \cite{minimal}.

\begin{thm}\label{iag}
Suppose there is a strategy $\sigma$ for I in $F_X$ to play in the set
$\B\subseteq E^\infty$. Then for any sequence $\Delta>0$ there are intervals
$I_0<I_1<I_2<\ldots$ of $\N$ such that for any block sequence $(x_i)\in \go
B(X)$, if
$$
\a n\;\e m\; I_0<x_n<I_m<x_{n+1},
$$
then $(x_i)\in \B_\Delta$.
\end{thm}

\begin{proof}Choose sets
$\D_n\subseteq \ku B_X$  such that  for each finite $d\subseteq \N$, the number
of $x\in \D_n$ such that ${\rm supp}\; x=d$ is finite, and for every $x\in \ku
B_X$ there is some $y\in \D_n$ with ${\rm supp} \;x={\rm supp}\; y$ and
$\norm{x-y}<\delta_n$. This is possible since the unit ball in $[e_i]_{i\in d}$
is totally bounded for all finite $d\subseteq\N$.

For each position $p=(n_0,y_0,\ldots,n_i,y_i)$ in $F_X$ played according to
$\sigma$ in which $y_j\in \D_j$ for all $j$, we write $p<k$ if $n_j,y_j<k$ for
all $j$. Notice that for all $k$ there are only finitely many such $p$ with
$p<k$, so we can define
$$
\alpha(k)=\max(k,\max\{\sigma(p)\del p<k\})
$$
and set $I_k=[k,\alpha(k)]$. The $I_k$ are not necessarily successive, but
their minimal elements tend to $\infty$. So, modulo passing to a subsequence,
it is enough to show that if $(x_i)\in \go B(X)$ and
$$
\a n\;\e m\; I_0<x_n<I_m<x_{n+1},
$$
then $(x_i)\in \B_\Delta$.

Suppose such $(x_i)$ is given. Find $y_i\in \D_i$ such that
$\norm{x_i-y_i}<\delta_i$ and ${\rm supp}\;x_i={\rm supp}\; y_i$ for all $i$
and let $0=b_0<b_1<b_2<\ldots$ be integers such that
$$
I_{b_0}<y_0<I_{b_1}<y_1<I_{b_2}<y_2<\ldots.
$$
We claim that there are natural numbers $n_i\leq \max I_{b_i}$ such that each
$$
p_i=(n_0,y_0,\ldots, n_i,y_i)
$$
is a  position in $F_X$ in which I has played according to $\sigma$. To see
this, notice first that $n_0=\alpha(\tom)\in I_{b_0}$, so $p_0=(n_0,y_0)$ is
played according to $\sigma$. Now, for the induction step, suppose that $p_i$
is played according to $\sigma$, and notice that $p_i<\min
I_{b_{i+1}}=b_{i+1}$. We set $n_{i+1}=\sigma(p_i)\leq \alpha(b_{i+1})=\max
I_{b+1}$, whereby $p_{i+1}$ is played according to $\sigma$. This finishes the
induction and proves the claim.

Thus, $(n_0,y_0,n_1,y_1,\ldots)$ is a run of the game in which I has followed
the strategy $\sigma$ and so $(y_i)\in\B$, whereby $(x_i)\in \B_\Delta$.
\end{proof}

The following result is a slight variant of the central result of Gowers' paper
\cite{gowers}. The variation, which is insignificant for applications, lies in
the fact that the $\Delta$-approximations appear on the opposite side of the
dichotomy. If one instead wants the approximations on the other side of the
dichotomy, and hence get the exact same statement as in \cite{gowers}, one can
just apply Theorem \ref{dicho} to $\B=\A_\Delta$ instead of $\A$ itself.

A set $\B\subseteq E^\infty$ is said to be {\em large} if for all $X\subseteq
E$, $\B\cap \go B(X)\neq \tom$. Also, let
$$
{\rm Int}_\Delta(\B)=\sim(\sim\B)_\Delta=\{(x_i)\del \a (z_i)\;(\a i\; \norm{x_i-z_i}<\delta_i\til (z_i)\in \B)\}.
$$

\begin{thm}\label{dicho}
Suppose $\A\subseteq E^\infty$ is strategically Ramsey and for some $\Delta>0$,
${\rm Int}_\Delta(\A)$ is large. Then there is $X\subseteq E$ such that II has
a strategy in $G_X$ to play in $\A$.
\end{thm}

\begin{proof}
Suppose for a contradiction that for some $X\subseteq E$, I has a strategy in
$F_X$ to play in $\sim\A=E^\infty\setminus\A$. Then, using Theorem \ref{iag},
we can find some $Y\subseteq X$ such that $\go B(Y)\subseteq (\sim\A)_\Delta$,
contradicting that ${\rm Int}_\Delta(\A)$ is large. So since $\A$ is
strategically Ramsey there is instead $X\subseteq E$ such that II has a
strategy in $G_X$ to play in $\A$.
\end{proof}

Suppose $X\subseteq E$. We define  {\em Gowers' unraveled game} $H_X$ played
below $X$ between two players I and II as follows: I and II alternate (with I
beginning) in choosing  infinite dimensional subspaces  $Y_0, Y_1,
Y_2,\ldots\subseteq X$, respectively  vectors $x_0<x_1<x_2<\ldots$ and digits
$\eps_i\in \{0,1\}$, according to the constraint $x_i\in Y_i$.
$$
\begin{array}{cccccccccccc}
{\bf I} & & & Y_0 &  & Y_1 &  & Y_2 & & Y_3& &\ldots \\
{\bf II} & & &  & x_0,\eps_0 &  & x_1,\eps_1 &  & x_2,\eps_2 & & x_3,\eps_3 &\ldots
\end{array}
$$
We say that the pair of sequences $((x_n)_{n\in \N},(\eps_n)_{n\in \N})$ is the
{\em outcome} of the game.

The following  result is exceedingly useful in applications.
\begin{thm}
Let $\B\subseteq E^\infty\times 2^\infty$ be analytic such that $\A={\rm
proj}_{E^\infty}(\B)$ is large. Then for every $\Delta>0$ there is $X\subseteq
E$ such that  II has a strategy in $H_X$ to play in
$$
\B_\Delta=\{((y_n),(\eps_n))\del \e (x_n)\;\a
n\;\norm{y_n-x_n}<\delta_n\;\&\;((x_n),(\eps_n))\in \B\}.
$$
\end{thm}

\begin{proof}We can suppose that $\frac14>\delta_0>\delta_1>\ldots$. Also, for
simplicity of notation, let us suppose temporarily that $2$ is the set
$\{\frac12,1\}$, so $\B\subseteq E^\infty\times \{\frac12,1\}^\infty$.  Define
$\D\subseteq E^\infty$ as follows:
$$
\D=\{(x_i)\in E^\infty\del ((x_{2i})_{i=0}^\infty,(\norm{x_{2i+1}})_{i=0}^\infty)\in \B\}.
$$
We claim that $\D$ is large. For suppose $X\subseteq E$ is spanned by a block
sequence $(z_n)$, let $Z=[z_{2n}]$ and find some $((y_n),(\eps_n))\in \B$ such
that $(y_n)\in\A\cap\go B(Z)$. Now for all $n$, find some $v_n\in [z_{2n+1}]$
such that $y_n<v_n<y_{n+1}$ and $\norm{v_n}=\eps_n$. (This is where we use that
the norm takes values in $\go F$ and hence that we can normalise). Then
$(y_0,v_0,y_1,v_1,\ldots)\in\D\cap\go B(X)$, verifying the largeness of $\D$.
Since $\D\subseteq {\rm Int}_\Delta(\D_\Delta)$ and $\D_\Delta$ is analytic, by
Theorem \ref{dicho} there is some $X\subseteq E$ such that II has a strategy in
$G_X$ to play in $\D_\Delta$. Since for $(x_i)\in \D$, $\norm{x_{2i+1}}$ is
either $1$ or $\frac12$ and moreover $\delta_{2i+1}<\frac14$, this easily
implies that II has a strategy in $G_X$ to play in $\D_\Delta$ such that
moreover $\norm{x_{2i+1}}$ is either $1$ or $\frac 12$ for all $i$. Using this,
II evidently has a strategy in $H_X$ to play in $\B_\Delta$.
\end{proof}

\section{Strategically Ramsey sets under set theoretical hypotheses}

\begin{thm}\label{ctbl unions}
The class of strategically Ramsey sets is closed under countable unions.
\end{thm}

\begin{proof}
Let $\A_n$ be strategically Ramsey for every $n$ and set $\B=\bigcup_n\A_n$.
Let $\vec x$ and $X\subseteq E$ be given. Since each $\A_n$ is strategically
Ramsey, by diagonalising, there is some $Y\subseteq X$ such that for all $\vec
y$ and $n$, either II has a strategy in $F_Y(\vec y)$ to play in $\A_n$ or I
has a strategy in $G_Y(\vec y)$ to play in $\sim \A_n$. Also, by Lemma
\ref{ctbl split} there is $Z\subseteq Y$ such that either
\begin{itemize}
  \item [(a)] II has a strategy in $G_Z$ to play $(z_i)$ such that
  $$
 \e n\;\a V\subseteq Z\; \textrm{I has no strategy in $F_V(\vec x\con (z_0,\ldots,z_n))$ to play in $\sim\A_n$},
  $$
  or
  \item[(b)] I has a strategy in $F_Z(\vec x)$ to play in $\sim \B$.
\end{itemize}
Note that (a) implies that II has a strategy in $G_Z$ to play $(z_i)$ such that
  $$
 \e n\;\textrm{II has a strategy in $G_Z(\vec x\con (z_0,\ldots,z_n))$ to play in $\A_n$}.
  $$
And, in this case, II first follows the strategy to play some
$(z_0,\ldots,z_n)$ such that II has a strategy in $G_Z(\vec x\con
(z_0,\ldots,z_n))$ to play in $\A_n$ and thereafter continues with this other
strategy. This, combined, is a strategy for II in $G_Z(\vec x)$ to play in
$\B=\bigcup_m\A_m$.
\end{proof}

\begin{thm}[$\textrm{MA}_{\om_1}$]\label{unctbl unions}
A union of $\aleph_1$ many strategically Ramsey sets is again strategically
Ramsey.
\end{thm}

\begin{proof}
By Theorem \ref{ctbl unions}, it is enough to consider well-ordered increasing
unions of length $\om_1$. So suppose $\A_\xi\subseteq\A_\zeta\subseteq
E^\infty$ are strategically Ramsey for all $\xi<\zeta<\om_1$ and
$\B=\bigcup_{\zeta<\om_1}\A_\zeta$. Fix $\vec x$ and $X\subseteq E$. Since
every $\A_\xi$ is strategically Ramsey, we can define a decreasing sequence
$\ldots\subseteq^*X_\xi\subseteq^*\ldots\subseteq^* X_2\subseteq^*
X_1\subseteq^*X_0\subseteq X$ of length $\om_1$ such that for all $\xi<\om_1$
either
\begin{itemize}
  \item [(a)] II has a strategy in $G_{X_\xi}(\vec x)$ to play in $\A_\xi$,
  or
  \item [(b)] I has a strategy in $F_{X_\xi}(\vec x)$ to
  play in $\sim\A_\xi$.
\end{itemize}
If for some $\xi$, II has a strategy in $G_{X_\xi}(\vec x)$ to play in
$\A_\xi$, then II also has a strategy in $G_{X_\xi}(\vec x)$ to play in
$\B=\bigcup_{\zeta<\om_1}\A_\zeta$ and we are done. So suppose instead that for
every $\xi$, I has a strategy in $F_{X_\xi}(\vec x)$ to play in $\sim\A_\xi$.
By Lemma 5 in \cite{ergodic}, under $\textrm{MA}_{\om_1}$ there is a
$Y\subseteq X$ such that $Y\subseteq^*X_\xi$ for all $\xi$. Thus, for every
$\xi$, I has a strategy $\sigma_\xi$ in $F_{Y}(\vec x)$ to play in
$\sim\A_\xi$.

Notice that $\sigma_\xi$ is formally a function from the countable set $D$ of
finite block sequences $\vec y$ of $Y$ to the set of natural numbers and hence
a member of $\N^D$. By $\textrm{MA}_{\om_1}$, the family
$\{\sigma\}_{\xi<\om_1}$ cannot be $\leq^*$ unbounded in $\N^D$ and hence for
some $\sigma\in \N^D$ we have $\sigma_\xi\leq^*\sigma$ for all $\xi$, i.e., for
all $\xi$ there is a finite set $p_\xi\subseteq D$ such that
$$
\a \vec y\in D\setminus p_\xi\; \;\;\sigma_\xi(\vec y)\leq \sigma(\vec y).
$$
By reason of cardinality, there is some $p\subseteq D$ such that for an
unbounded set $S\subseteq \om_1$ we have $p_\xi=p$ for all $\xi\in S$. Now let
$n_0$ be large enough such that $n_0\nless y_0$ for all $\vec
y=(y_0,\ldots,y_m)\in p$. We modify $\sigma$ so that $\sigma(\tom)=n_0$ and
otherwise leave it unaltered. Then $\sigma$ is a strategy for I in $F_Y(\vec
x)$ to play in $\sim \B=\bigcap_{\xi<\om_1}\sim\A_\xi=\bigcap_{\xi\in
S}\sim\A_\xi$. To see this, suppose that $(z_i)$ is the outcome of a game in
which I has followed $\sigma$. Then as $n_0<z_0$, we must have
$(z_0,\ldots,z_m)\notin p$ for all $m$, and hence for all $\xi\in S$ and $m$,
$\sigma(z_0,\ldots,z_m)=\sigma_\xi(z_0,\ldots,z_m)$. If follows that for every
$\xi\in S$, I has followed the strategy $\sigma_\xi$ and hence $(z_i)\notin
\A_\xi$.
\end{proof}

Since ${\bf \Sigma}_2^1$ sets are unions of $\aleph_1$ many Borel sets, we have
the following strengthening of a result of Bagaria and L\'opez-Abad
\cite{jordi}. They essentially proved the conclusion of Theorem \ref{dicho} for
${\bf \Sigma}_2^1$ sets, but only under a hypothesis relatively consistent with
the existence of a large cardinal.  On the other hand, our hypothesis, namely
$\textrm{MA}_{\om_1}$, is equiconsistent with ZF, which permits the use of
absoluteness arguments.

\begin{cor}[$\textrm{MA}_{\om_1}$]
${\bf \Sigma}_2^1$ sets are strategically Ramsey.
\end{cor}

We do not know if the axiom of projective determinacy suffices to prove that
all projective sets are strategically Ramsey, though we very much suspect so.
Again, Bagaria and L\'opez-Abad \cite{joan} proved that under PD, projective
sets are weakly Ramsey.

\section{Relational games}
In this section we consider relational versions of Gowers' game and the
infinite asymptotic game in which both players contribute to the outcome.
Unfortunately, we can in this case only prove the Ramsey principle for open and
closed sets. Simpler relational games were first considered by A. M. Pe\l czar
\cite{anna} in connection with subsymmetry of block sequences.

\

Suppose $X\subseteq E$. We define the game $A_X$ played below $X$ between two
players I and II as follows: I and II alternate  in choosing block subspaces
$Z_0,Z_1,Z_2,\ldots\subseteq X$ and vectors $x_0<x_1<x_2<\ldots\in X$,
respectively integers $n_0<n_1<n_2<\ldots$ and vectors $y_0<y_1<y_2<\ldots\in
X$ according to the constraints $n_i< x_i$ and $y_i\in Z_i$:
$$
\begin{array}{cccccccccccc}
{\bf I} & & &n_0<x_0, Z_0 &  & n_1<x_1,Z_1 &  & n_2<x_2,Z_2 &\ldots \\
{\bf II} & &n_0 &  & y_0\in Z_0,n_1 &  & y_1\in Z_1,n_2 &  &\ldots
\end{array}
$$
We say that the sequence $(x_0,y_0,x_1,y_1,\ldots)$ is the {\em outcome} of the
game.

If $\vec x$ is a finite block sequence of {\em even} length, the game $A_X(\vec
x)$ is defined as above except that the outcome is now $\vec x\con
(x_0,y_0,x_1,y_1,\ldots)$.

On the other hand, if $\vec x$ is a finite block sequence of {\em odd} length,
$A_X(\vec x)$ is defined in a similar way as before except that I begins
$$
\begin{array}{cccccccccccc}
{\bf I} & & Z_0 &  & n_0<x_0,Z_1 &  & n_1<x_1,Z_2 &&\ldots \\
{\bf II} &  &  & y_0\in Z_0,n_0 &  & y_1\in Z_1,n_1 &  &y_2\in Z_2,n_2&\ldots
\end{array}
$$
and the {\em outcome} is now $\vec x\con(y_0,x_0,y_1,x_1,\ldots)$ rather than
$\vec x\con(x_0,y_0,x_1,y_1,\ldots)$.

\

We define the game $B_X$ in a similar way to $A_X$ except that we now have I
playing integers and II playing block subspaces:
$$
\begin{array}{cccccccccccc}
{\bf I} &  &  & x_0\in Z_0,n_0 &  & x_1\in Z_1,n_1 &  &x_2\in Z_2,n_2&\ldots       \\
{\bf II} & & Z_0 &  & n_0<y_0,Z_1 &  & n_1<y_1,Z_2 &&\ldots
\end{array}
$$
with $x_i\in Z_i\subseteq X$ and $n_i<y_i\in X$. Again, the {\em outcome} is
$(x_0,y_0,x_1,y_1,\ldots)$.

If $\vec x$ is a finite block sequence of {\em even} length, the game $B_X(\vec
x)$ is defined as above except that the outcome is now $\vec x\con
(x_0,y_0,x_1,y_1,\ldots)$.

On the other hand, if $\vec x$ is a finite block sequence of {\em odd} length,
$B_X(\vec x)$ is defined by letting I begin
$$
\begin{array}{cccccccccccc}
{\bf I} & &n_0 &  & x_0\in Z_0,n_1 &  & x_1\in Z_1,n_2 &  &\ldots    \\
{\bf II} & & &n_0<y_0, Z_0 &  & n_1<y_1,Z_1 &  & n_2<y_2,Z_2 &\ldots
\end{array}
$$
and the {\em outcome} is now $\vec x\con (y_0,x_0,y_1,x_1,\ldots)$.

\

Thus, in both games $A_X$ and $B_X$, one should remember that I is the {\em
first} to play a vector. And in $A_X$, I plays block subspaces and II plays
tail subspaces, while in $B_X$, II takes the role of playing block subspaces
and I plays tail subspaces.

Suppose $\A\subseteq E^\infty$, $Y\subseteq^* X$ and $\vec x$ are given. Then
one easily sees that if II has a strategy in $A_X(\vec x)$ to play in $\A$,
then II also has a strategy in $A_Y(\vec x)$ to play in $\A$. Similarly, if I
has a strategy in $B_X(\vec x)$ to play in $\A$, then I also has a strategy in
$B_Y(\vec x)$ to play in $\A$. Also, if II has a strategy in $A_X(\vec x)$ to
play in $\A$, then II also has a strategy in $B_X(\vec x)$ to play in $\A$.

\begin{thm}
Suppose $\A\subseteq E^\infty$ is open or closed. Then there is $X\subseteq E$
such that either
\begin{enumerate}
  \item II has a strategy in $A_X$ to play in $\A$, or
  \item I has a strategy in $B_X$ to play in $\sim\A$.
\end{enumerate}
\end{thm}

\begin{proof}Suppose first that $\A$ is open.
We say that
\begin{itemize}
  \item [(a)] $(\vec x,X)$ is {\em good} if II has a strategy in $A_X(\vec x)$ to play in
  $\A$.
  \item [(b)] $(\vec x,X)$ is {\em bad} if $\a  Y\subseteq X$, $(\vec x,Y)$ is not good.
  \item [(c)] $(\vec x,X)$ is {\em worse} if it is bad and either
  \begin{enumerate}
    \item $|\vec x|$ is odd and $\e n\;\a y\in X\;(n<y\til (\vec x\con y,X) \textrm{ is
    bad})$, or
    \item $|\vec x|$ is even and $\a Y\subseteq X\;\e x\in Y\;(\vec x\con x,X) \textrm{ is
    bad})$.
  \end{enumerate}
\end{itemize}
One checks as always that good, bad and worse are all $\subseteq^*$-hereditary.

\begin{lemme}
If $(\vec x,X)$ is bad, then there is some $Z\subseteq X$ such that $(\vec
x,Z)$ is worse.
\end{lemme}

\begin{proof}By diagonalisation, we can find some $Y\subseteq X$ such that for all $\vec y$,
$(\vec y,Y)$ is either good or bad.

Assume first that $|\vec x|$ is even. Since $(\vec x,Y)$ is bad, we have $\a
V\subseteq X$ II has no strategy in $A_V(\vec x)$ to play in $\A$. So $\a
V\subseteq X\;\e x\in V$ such that II has no strategy in $A_V(\vec x\con x)$ to
play in $\A$, and hence such that $(\vec x\con x,V)$ is not good. Thus,
$$
\a V\subseteq X\;\e x\in V\; (\vec x\con x,Y) \textrm{ is bad},
$$
and so already $(\vec x,Y)$ is worse.

Now suppose instead that $|\vec x|$ is odd and, towards a contradiction, that
there is no $Z\subseteq Y$ such that $(\vec x,Z)$ is worse. Then, as $(\vec
x,Y)$ is bad,  $\a Z\subseteq Y\; \e y\in Z\; (\vec x\con y,Z)$ is not bad and
thus also $\a Z\subseteq Y\; \e y\in Z\; (\vec x\con y,Y)$ is good. So
$$
\a Z\subseteq Y\; \e y\in Z\; \textrm{ II has a strategy in $A_Y(\vec x\con y)$ to play in $\A$},
$$
and hence II also has a strategy in $A_Y(\vec x)$ to play in $\A$,
contradicting that $(\vec x,Y)$ is bad.
\end{proof}

Diagonalising, we now find $X\subseteq E$ such that for all $\vec x$, either
$(\vec x,X)$ is good or worse. Assume that II has no strategy in $A_X$ to play
in $\A$, whereby $(\tom,X)$ is worse. Then, by unraveling the definition of
worse and using that bad and worse coincide below $X$, one sees that I has a
strategy in $B_X$ to produce block sequences $(z_0,z_1,z_2,\ldots)$ so that for
all $m$, $(z_0,z_1,\ldots,z_m,X)$ is worse. In particular, for no $m$ does II
have a strategy in $A_X(z_0,\ldots,z_m)$ to play in $\A$, and so, as $\A$ is
open, we must have $(z_0,z_1,z_2,\ldots)\in \sim\A$. So I has a strategy in
$B_X$ to play in $\sim\A$, which finishes the proof for open sets.

Now if instead $\A$ is closed, set
$$
\B=\{x\con {\bf x}\del x\in E\;\&\; {\bf x}\notin \A\}=E\times \sim\A,
$$
which is open. So find some $X\subseteq E$ such that either
\begin{enumerate}
  \item II has a strategy in $A_X$ to play in $\B$, or
  \item I has a strategy in $B_X$ to play in $\sim\B$.
\end{enumerate}
Now if  II has a strategy in $A_X$ to play in $\B$, then I has a strategy in
$B_X$ to play in $\sim\A$. And if I has a strategy in $B_X$ to play in
$\sim\B$, then II has a strategy in $A_X$ to play in $\A$, which is what needed
proof.
\end{proof}

\

\

\begin{flushleft}
{\em Address of C. Rosendal:}\\
Department of Mathematics, Statistics, and Computer Science\\
University of Illinois at Chicago\\
322 Science and Engineering Offices (M/C 249)\\
851 S. Morgan Street\\
Chicago, IL 60607-7045.
\end{flushleft}

\end{document}